\documentclass{article}

\usepackage{arxiv}
\usepackage[utf8]{inputenc} 
\usepackage[T1]{fontenc}    
\usepackage{hyperref}       
\usepackage{url}            
\usepackage{booktabs}       
\usepackage{amsfonts}       
\usepackage{nicefrac}       
\usepackage{microtype}      
\usepackage{lipsum}	
\usepackage{amsthm}
\usepackage{amsmath}

\title{Dirichlet product of derivative arithmetic with an arithmetic function multiplicative}

\author{
  Es-said En-naoui\\
 \texttt{essaidennaoui1@gmail.com} \\}
\newtheorem{theorem}{Theorem}

\newtheorem{lemma}[theorem]{Lemma}
\newtheorem{proposition}[theorem]{Proposition}
\newtheorem{definition}[theorem]{Definition}
\newtheorem{example}[theorem]{Example}

\begin{document}
\maketitle

\begin{abstract}
We define the derivative of an integer to be the map sending every prime to 1 and satisfying the Leibniz rule. The aim of this article is to calculate the Dirichlet product of
this map with a function arithmetic multiplicative.
\end{abstract}

\section{Introduction}
Barbeau \cite{Barb} defined the arithmetic derivative as the
function $\delta\;:\; \mathbb N \rightarrow \mathbb N$ ,$\;$defined by the rules : 
\begin{enumerate}
\item  $\delta(p) = 1$ for any prime $p \in \mathbb P := \{2,3,5,7,\ldots, p_{i},\ldots \}$.
\item $\delta(ab) = \delta(a)b + a\delta(b)$ for any $a,b \in \mathbb N$ (the Leibnitz rule) .
\end{enumerate}

\vspace{0.4cm}
Let $n$ a positive integer , if $n=\prod_{i=1}^{s}p_{i}^{\alpha_{i}}$ is the prime factorization of $n$, then
the formula for computing the arithmetic derivative
of n is (see, e.g., \cite{Barb, Ufna}) giving by :
\begin{equation}
\delta(n)=n\sum \limits_{i=1}^s\frac{\alpha_i}{p_i}
=n\sum \limits_{p^{\alpha}||n}\frac{\alpha}{p}
\end{equation}
A brief summary on the history of arithmetic derivative and its generalizations
to other number sets can be found, e.g., in \cite{Stay}  .
\vspace{0.5cm}\\
First of all, to cultivate analytic number theory one must acquire a considerable 
skill for operating with arithmetic functions. We begin with a few elementary  considerations. 

\begin{definition}[arithmetic function]
An \textbf{arithmetic function} is a function $f:\mathbb{N}\longrightarrow \mathbb{C}$ with
domain of definition the set of natural numbers $\mathbb{N}$ and range a subset of the set of complex numbers $\mathbb{C}$.
\end{definition}

\begin{definition}[multiplicative function]
A function $f$ is called an \textbf{multiplicative function} if and
only if : 
\begin{equation}\label{eq:1}
f(nm)=f(n)f(m)
\end{equation}
for every pair of coprime integers $n$,$m$. In case (\ref{eq:1}) is satisfied for every pair of integers $n$ and $m$ , which are not necessarily coprime, then the function $f$ is
called \textbf{completely multiplicative}.
\end{definition}

Clearly , if $f$ are a multicative function , then
 $f(n)=f(p_1^{\alpha _1})\ldots f(p_s^{\alpha _s})$, 
 for any positive integer $n$ such that 
$n = p_1^{\alpha _1}\ldots  p_s^{\alpha _s}$ , and if $f$ is completely multiplicative , so we have : $f(n)=f(p_1)^{\alpha _1}\ldots f(p_s)^{\alpha _s}$.

\begin{example}
Let $n\in\mathbb{N}^*$  , This is the same classical arithmetic functions used in this article  :
\begin{enumerate}

\item \textbf{Identity function }: The function defined by $Id(n)=n$ for all positive integer n. 

\item \textbf{The Euler phi function}  
 :\;$\varphi(n)=\sum \limits_{\underset{gcd(k,n)=1}{k=1}}^n 1$.

\item \textbf{The number of distinct prime divisors of n} :\;$\omega(n)=\sum \limits_{p|n}1$ .

 \item \textbf{The Mobiuse function } 
 :\;$\mu(n)=\left\{ \begin{array}{cl}
1 & \textrm{if }\;\;n=1\\
0& \textrm{if }\;\; p^2|n\;\; for\; some\; prime\; p\\
(-1)^{\omega(n)} & \textrm{otherwise}\\
\end{array}\right.$

 \item \textbf{number of positive divisors of $n$} defined by :  $\tau(n)=\sum \limits_{d|n}1$ .
 \item \textbf{sum of divisors function of $n$} defined by :  $\sigma(n)=\sum \limits_{d|n}d$ .
\end{enumerate}

\end{example}

\vspace{0.4cm}
Now ,if $f,g:\mathbb{N}\longrightarrow \mathbb{C}$ are two arithmetic functions from the positive integers to the complex numbers, the Dirichlet convolution $f*g$ is a new arithmetic function defined by: 
\begin{equation}
(f*g)(n)=\sum \limits_{d|n}f(d)g(\frac{n}{d})=
\sum \limits_{ab=n}f(a)g(b)
\end{equation}
where the sum extends over all positive divisors $d$ of $n$ , or equivalently over all distinct pairs $(a, b)$ of positive integers whose product is $n$.\\
In particular, we have $(f*g)(1)=f(1)g(1)$ ,$(f*g)(p)=f(1)g(p)+f(p)g(1)$ for any prime $p$ and for any power prime $p^m$ we have :
\begin{equation}
(f*g)(p^m)=\sum \limits_{j=0}^m f(p^j)g(p^{m-j})
\end{equation} 
This product occurs naturally in the study of Dirichlet series such as the Riemann zeta function. It describes the multiplication of two Dirichlet series in terms of their coefficients: 
\begin{equation}\label{eq:5}
\bigg(\sum \limits_{n\geq 1}\frac{\big(f*g\big)(n)}{n^s}\bigg)=\bigg(\sum \limits_{n\geq 1}\frac{f(n)}{n^s} \bigg)
\bigg( \sum \limits_{n\geq 1}\frac{g(n)}{n^s} \bigg)
\end{equation}
with Riemann zeta function or  is defined by : $$\zeta(s)= \sum \limits_{n\geq 1} \frac{1}{n^s}$$
These functions are widely studied in the literature (see, e.g., \cite{book1, book2, book3}).\\

Now before to past to main result we need this propriety , if $f$ and $g$ are multiplicative function , then $f*g$ is multiplicative.

\section{Main results }

In this section we give the new result of Dirichlet product between derivative arithmetic and an arithmetic function multiplicative $f$ , and we will give the relation between $\tau$ and the derivative arithmetic .

\begin{theorem}
Given a multiplicative function $f$, and lets $n$ and $m$ two positive integers such that $gcd(n,m)=1$ , Then  we have :  
\begin{equation}
(f*\delta)(nm)=\big(Id*f\big)(n).\big(f*\delta\big)(m)+\big(Id*f\big)(m).\big(f*\delta\big)(n)
\end{equation}
\end{theorem}

\begin{proof}
Lets $n$ and $m$ two positive integers such that $gcd(n,m)=1$, and let $f$  an arithmetic function multiplicative , then we have :
\begin{align*}
(f*\delta)(nm) &=\sum \limits_{d|nm} f\big(\frac{nm}{d}\big)\delta(d)
=\sum \limits_{\underset{d_2|m}{d_1|n}} f(\frac{nm}{d_1d_2})\delta(d_1d_2)
=\sum \limits_{\underset{d_2|m}{d_1|n}} f(\frac{n}{d_1}) f(\frac{m}{d_2}) \bigg(d_1\delta(d_2)+d_2\delta(d_1)\bigg)
\\
&=\sum \limits_{\underset{d_2|m}{d_1|n}}\bigg(  d_1f(\frac{n}{d_1}) f(\frac{m}{d_2})\delta(d_2)
+d_2f(\frac{m}{d_2})f(\frac{n}{d_1})\delta(d_1)
\bigg)
\\
& =\bigg(\sum \limits_{d_1|n} d_1f(\frac{n}{d_1})\bigg)\bigg(\sum \limits_{d_2|m} f(\frac{m}{d_2})\delta(d_2)\bigg)
+
\bigg(\sum \limits_{d_2|m} d_2f(\frac{m}{d_2})\bigg)\bigg(\sum \limits_{d_1|n} f(\frac{n}{d_1})\delta(d_1)\bigg)
\\
&=\big(Id*f\big)(n).\big(f*\delta\big)(m)+\big(Id*f\big)(m).\big(f*\delta\big)(n)
\end{align*}
\end{proof}

\begin{lemma}
For any natural number $n$ , if $n=\prod_{i=1}^{s}p_{i}^{\alpha_{i}}$ is the prime factorization of n, then :
\begin{equation}
\big(f*\delta\big)(n)=\big(Id*f\big)(n)\sum \limits_{i=1}^s
\frac{\big(f*\delta\big)(p^{\alpha_i}_i)}{\big(Id*f\big)(p^{\alpha_i}_i)}
\end{equation}
\end{lemma}

\begin{proof}
Let $n$ a positive integer such that $n = p_1^{\alpha _1}\ldots  p_s^{\alpha _s}$ and let $f$ an arithmetic function , Then :
\begin{align*}
\big(f*\delta\big)(n) &=\big(f*\delta\big)(p_1^{\alpha _1}\ldots  p_s^{\alpha _s})
\\
&=
\big(Id*f\big)(p_2^{\alpha _2}\ldots  p_s^{\alpha _s}).\big(f*\delta\big)(p^{\alpha_1}_1)
+
\big(Id*f\big)(p^{\alpha_1}_1).\big(f*\delta\big)(p_2^{\alpha _2}\ldots  p_s^{\alpha _s})
\\
& =\big(Id*f\big)(n).
\frac{\big(f*\delta\big)(p^{\alpha_1}_1)}{\big(Id*f\big)(p_1^{\alpha _1})}
+
\big(Id*f\big)(p^{\alpha_1}_1).
\bigg[
\big(Id*f\big)(p_3^{\alpha _3}\ldots  p_s^{\alpha _s}).\big(f*\delta\big)(p^{\alpha_2}_2)
+\\
&+
\big(Id*f\big)(p^{\alpha_2}_2).\big(f*\delta\big)(p_3^{\alpha _3}\ldots  p_s^{\alpha _s})
\bigg]
\\
&=\big(Id*f\big)(n)\frac{\big(f*\delta\big)(p^{\alpha_1}_1)}{\big(Id*f\big)(p_1^{\alpha _1})}
+
\big(Id*f\big)(n)\frac{\big(f*\delta\big)(p^{\alpha_2}_2)}{\big(Id*f\big)(p_2^{\alpha _2})}
+\\
&+
\big(Id*f\big)(p^{\alpha_1}_1)\big(Id*f\big)(p^{\alpha_2}_2)\big(f*\delta\big)(p_3^{\alpha _3}\ldots  p_s^{\alpha _s})
\\
\vdots\\
&=\big(Id*f\big)(n)\frac{\big(f*\delta\big)(p^{\alpha_1}_1)}{\big(Id*f\big)(p_1^{\alpha _1})}
+
\big(Id*f\big)(n)\frac{\big(f*\delta\big)(p^{\alpha_2}_2)}{\big(Id*f\big)(p_2^{\alpha _2})}
+ \ldots
+
\big(Id*f\big)(n)\frac{\big(f*\delta\big)(p^{\alpha_s}_s)}{\big(Id*f\big)(p_s^{\alpha _s})}
\\
&=\big(Id*f\big)(n)
\bigg[
\frac{\big(f*\delta\big)(p^{\alpha_1}_1)}{\big(Id*f\big)(p_1^{\alpha _1})}+\ldots+ \frac{\big(f*\delta\big)(p^{\alpha_s}_s)}{\big(Id*f\big)(p_s^{\alpha _s})}
\bigg]\\
&=\big(Id*f\big)(n)\sum \limits_{i=1}^s
\frac{\big(f*\delta\big)(p^{\alpha_i}_i)}{\big(Id*f\big)(p^{\alpha_i}_i)}
\end{align*}
\end{proof}

an other prof by induction on $s$ that if $n=\prod_{i=1}^{s}p_{i}^{\alpha_{i}}$ then $(f*\delta)(n)=\big(Id*f\big)(n)\sum \limits_{i=1}^s
\frac{\big(f*\delta\big)(p^{\alpha_i}_i)}{\big(Id*f\big)(p^{\alpha_i}_i)}$.

\begin{proof}
Consider $n\in\mathbb{N}$ and express $n=\prod_{i=1}^{s}p_{i}^{\alpha_{i}}$ where all $p_i$ are distinct .\\
where $s=1$  ,  it is clear that $(f*\delta)(n)=\big(Id*f\big)(n)\sum \limits_{i=1}^1
\frac{\big(f*\delta\big)(p^{\alpha_i}_i)}{\big(Id*f\big)(p^{\alpha_i}_i)}
=\big(Id*f\big)(p^{\alpha_1})
\frac{\big(f*\delta\big)(p^{\alpha_1}_1)}{\big(Id*f\big)(p^{\alpha_1}_1)}=\big(f*\delta\big)(p^{\alpha_1}_1).
$\\
 Assume that $n=\prod_{i=1}^{s}p_{i}^{\alpha_{i}}$, then we have :
\begin{align*}
\big(id*\delta\big)(n.p^{\alpha_{s+1}}_{s+1})
 &=\big(Id*f\big)(p^{\alpha_{s+1}}_{s+1}).\big(f*\delta\big)(n)+\big(Id*f\big)(n).\big(f*\delta\big)(p^{\alpha_{s+1}}_{s+1})\\
 &=\big(Id*f\big)(p^{\alpha_{s+1}}_{s+1}).
 \big(Id*f\big)(n)\sum \limits_{i=1}^s
\frac{\big(f*\delta\big)(p^{\alpha_i}_i)}{\big(Id*f\big)(p^{\alpha_i}_i)}
+
\big(Id*f\big)(p^{\alpha_{s+1}}_{s+1}).
 \big(Id*f\big)(n)
\frac{\big(f*\delta\big)(p^{\alpha_{s+1}}_{s+1})}{\big(Id*f\big)(p^{\alpha_{s+1}}_{s+1})}\\
&=\big(Id*f\big)(n.p^{\alpha_{s+1}}_{s+1})\sum \limits_{i=1}^s
\frac{\big(f*\delta\big)(p^{\alpha_i}_i)}{\big(Id*f\big)(p^{\alpha_i}_i)}
+
\big(Id*f\big)(n.p^{\alpha_{s+1}}_{s+1})
\frac{\big(f*\delta\big)(p^{\alpha_{s+1}}_{s+1})}{\big(Id*f\big)(p^{\alpha_{s+1}}_{s+1})}\\
&=\big(Id*f\big)(n.p^{\alpha_{s+1}}_{s+1})\bigg[\sum \limits_{i=1}^s
\frac{\big(f*\delta\big)(p^{\alpha_i}_i)}{\big(Id*f\big)(p^{\alpha_i}_i)}
+
\frac{\big(f*\delta\big)(p^{\alpha_{s+1}}_{s+1})}{\big(Id*f\big)(p^{\alpha_{s+1}}_{s+1})}\bigg]\\
&=\big(Id*f\big)(n.p^{\alpha_{s+1}}_{s+1})\sum \limits_{i=1}^{s+1}
\frac{\big(f*\delta\big)(p^{\alpha_i}_i)}{\big(Id*f\big)(p^{\alpha_i}_i)}
\end{align*}
\end{proof}

\begin{proposition}\label{prop_6}
Let $f$ a function arithmetic multiplicative , and $\delta$ the derivative arithmetic , then we have :
\begin{equation}
\big(Id*\delta\big)(n)=\frac{1}{2}\tau(n)\delta(n) 
\end{equation}
\end{proposition}

\begin{proof}
Since $(Id*Id)(n)=\sum\limits_{d|n} \frac{n}{d}d
=n\sum \limits_{d|n} 1 =n\tau(n)$.\\
 and :\;\;
$(Id*\delta)(p^{\alpha})=
\sum \limits_{j=1}^{\alpha} \delta(p^j)Id(p^{\alpha-j})
=
\sum \limits_{j=1}^{\alpha} jp^{j -1}p^{\alpha-j}
=\frac{1}{2}\alpha(\alpha+1)p^{\alpha-1}
$.\\
Then for every a positive integer $n$ such that $n = p_1^{\alpha _1}\ldots  p_s^{\alpha _s}$ , we have :
\begin{align*}
\big(Id*\delta\big)(n)&=\big(Id*Id\big)(n)\sum \limits_{i=1}^s
\frac{\big(Id*\delta\big)(p^{\alpha_i}_i)}{\big(Id*Id\big)(p^{\alpha_i}_i)}\\
&=n\tau(n)\sum \limits_{i=1}^s
\frac{\frac{1}{2}\alpha_i(\alpha_{i}+1)p^{\alpha_i-1}_i}{p^{\alpha_i}_i\tau(p^{\alpha_i}_i)}\\
&=n\tau(n)\sum \limits_{i=1}^s
\frac{\frac{1}{2}\alpha_i(\alpha_{i}+1)p^{\alpha_i-1}_i}{p^{\alpha_i}_i(\alpha_i+1)}\\
&=\frac{1}{2} n\tau(n)\sum \limits_{i=1}^s \frac{\alpha_i}{p_i}=
\frac{1}{2}\tau(n)\delta(n)
\end{align*}
\end{proof}

So by the proposition \ref{prop_6} , and the equality \ref{eq:5} we have this relation between arithmetic derivative and the function $\tau$ :
\begin{equation}
2\zeta(s-1)\sum \limits_{n\geq 1} \frac{\delta(n)}{n^s}
=
\sum \limits_{n\geq 1} \frac{\delta(n)\tau(n)}{n^s}
\end{equation}

Let defined the new function arithmetic called \textbf{En-naoui} function , by : 
\begin{equation}
\Phi_\varphi(n)=n\sum \limits_{p|n}\bigg(1-\frac{1}{p}\bigg)
\end{equation}
Then we have this  equality related between 8 arithmetic function :
\begin{equation}
\big(\mu*\delta\big)(n)=
\varphi(n)\bigg(
\delta(n)-2\omega(n)+B(n)+\frac{\Phi_\varphi(n)}{n}+
\frac{\big(B*Id\big)(n)}{\sigma(n)}
\bigg) .
\end{equation}
with  $B$ is the arithmetic function defined by :
$B(n)=\sum \limits_{p^{\alpha}||n}\alpha p$.
In next article i will prove this equality, just I need to submit an article about this new function .

\section{Acknowledgements}

I would like to thank S.Allo for the opportunity to collaborate on my research and for his continued mentorship, and friendship, and I'm indebted to the Professor Zouhaïr Mouayn to endorse me to submit my article here.I also extend my deepest appreciation to my family for opening their minds, hearts, and homes to me and my work.

\bibliographystyle{unsrt}  


\end{document}